\documentclass[11pt]{article}

\usepackage[T1]{fontenc}
\usepackage{lmodern}

\usepackage{soul}
\usepackage{amsmath, amssymb, amsthm} 
\usepackage[margin=1in]{geometry}
\usepackage{color, graphicx}

\usepackage{etoolbox} 
\makeatletter
\patchcmd{\@maketitle}{\LARGE \@title}{\LARGE\bfseries\@title}{}{}

\renewcommand{\@seccntformat}[1]{\csname the#1\endcsname.\quad}
\makeatother

\definecolor{darkblue}{rgb}{0,0,.5}
\usepackage{hyperref}
\hypersetup{
	colorlinks=true,        % false: boxed links; true: colored links
	linkcolor=darkblue,     % color of internal links
	citecolor=darkblue,     % color of links to bibliography
	urlcolor=darkblue       % color of external links
}

\makeatletter
\def\th@plain{%
	\thm@notefont{}% same as heading font
	\itshape % body font
}
\def\th@definition{%
	\thm@notefont{}% same as heading font
	\normalfont % body font
}

\renewenvironment{proof}[1][\proofname]{\par
	\normalfont
	\topsep0\p@\@plus3\p@ \trivlist
	\item[\hskip\labelsep\itshape
	#1\@addpunct{.}]\ignorespaces
}{%
	\qed\endtrivlist
}
\makeatother

\newtheorem{theorem}{Theorem}[section]
\newtheorem{lemma}[theorem]{Lemma}
\newtheorem{corollary}[theorem]{Corollary}
\newtheorem{proposition}[theorem]{Proposition}

\theoremstyle{definition}

\theoremstyle{definition}

\theoremstyle{definition}
\newtheorem{remark}[theorem]{Remark}

\usepackage[shortlabels]{enumitem}
\setlist[enumerate]{nosep}

\newcommand{\menge}[2]{\{{#1}~\big |~{#2}\}} 
 
\newcommand{\Menge}[2]{\left\{{#1}~\Big|~{#2}\right\}} 

\newcommand{\scal}[2]{\left\langle {#1},{#2} \right\rangle}

\newcommand{\NN}{\ensuremath{\mathbb N}}
\newcommand{\nnn}{\ensuremath{{n\in{\mathbb N}}}}

\newcommand{\RR}{\ensuremath{\mathbb R}}

\newcommand{\RP}{\ensuremath{\mathbb{R}_+}}
\newcommand{\RPP}{\ensuremath{\mathbb{R}_{++}}}

\newcommand{\argmin}{\ensuremath{\operatorname*{argmin}}}

\newcommand{\ran}{\ensuremath{\operatorname{ran}}}
\newcommand{\zer}{\ensuremath{\operatorname{zer}}}
\newcommand{\dom}{\ensuremath{\operatorname{dom}}}

\newcommand{\gra}{\ensuremath{\operatorname{gra}}}
\newcommand{\Fix}{\ensuremath{\operatorname{Fix}}}
\newcommand{\Id}{\ensuremath{\operatorname{Id}}}
\newcommand{\prox}{\ensuremath{\operatorname{Prox}}}

\begin{document}

\title{Computing the resolvent of the sum of operators\\ with application to best approximation problems}

\author{
Minh N.\ Dao\thanks{CARMA, University of Newcastle, Callaghan, NSW 2308, Australia. 
E-mail: \texttt{daonminh@gmail.com}}
~and~
Hung M.\ Phan\thanks{Department of Mathematical Sciences, Kennedy College of Sciences, University of Massachusetts Lowell, Lowell, MA 01854, USA.
E-mail: \texttt{hung\char`_phan@uml.edu}.}}

\date{September 11, 2018}

\maketitle

\begin{abstract}
We propose a flexible approach for computing the resolvent of the sum of weakly monotone operators in real Hilbert spaces. This relies on splitting methods where strong convergence is guaranteed. We also prove linear convergence under Lipschitz continuity assumption. The approach is then applied to computing the proximity operator of the sum of weakly convex functions, and particularly to finding the best approximation to the intersection of convex sets.
\end{abstract}

{\small
\noindent{\bfseries AMS Subject Classifications:}
{Primary: 
47H05, % Monotone operators and generalization
49M27; % Decomposition methods
Secondary:
%41A25, % Rate of convergence, degree of approximation
47H10, % Fixed-point theorems
%47J25, % Iterative procedures
65K05, % Mathematical programming
65K10 % Optimization and variational techniques
}

\noindent{\bfseries Keywords:}
Best approximation,
Douglas--Rachford algorithm, 
linear convergence, 
operator splitting,
Peaceman--Rachford algorithm,
projector,
proximity operator,
resolvent,
strong convergence.
}

\section{Introduction}

In this paper, we explore a straightforward path to the problem of \emph{computing the resolvent of the sum of two (not necessarily monotone) operators} using resolvents of individual operators. When applied to normal cones of convex sets, this computation solves the \emph{best approximation problem} of finding the projection onto the intersection of these sets.

In general, computations involved simultaneously two or more operators are usually difficult. One popular approach is to treat each operator individually, then use these calculations to construct the desired answer. Prominent examples of such splitting strategy include the \emph{Douglas--Rachford algorithm} \cite{DR56,LM79} and the \emph{Peaceman--Rachford algorithm} \cite{PR55} that apply to the problem of finding a zero of the sum of maximally monotone operators. In \cite{BC08}, the authors proposed an extension of \emph{Dykstra's algorithm} \cite{Dyk83} for constructing the resolvent of the sum of two maximally monotone operators. By product space reformulation, this problem was then handled in \cite{Com09} for finitely many operators. 
Recently, the so-called \emph{averaged alternating modified reflections algorithm} was used in \cite{AC18} to study this problem, and was soon after re-derived in \cite{ABMW18} from the view point of the proximal and resolvent average. Because computing the resolvent of a finite sum of operators can be transformed into that of the sum of two operators by a standard product space setting, as done in \cite{Com09,AC18}, we will focus our consideration to the case of two operators for reason of clarity.

\emph{The goal of this paper is to provide a flexible approach for computing the resolvent of the sum of two weakly monotone operators from individual resolvents.} Our work extends and complements recent results in this direction. We also present applications to computing the proximity operator of the sum of two weakly convex functions and to finding the best approximation to the intersection of two convex sets. 

The paper is organized as follows. In Section~\ref{s:prep}, we provide necessary materials. Section~\ref{s:main} contains our main results. Finally, applications are presented in Section~\ref{s:appl}.

\section{Preparation}
\label{s:prep}

We assume throughout that $X$ is a real Hilbert space with inner product $\scal{\cdot}{\cdot}$ and induced norm $\|\cdot\|$. 
The set of nonnegative integers is denoted by $\NN$, the set of real numbers by $\RR$, the set of nonnegative real numbers by $\RP := \menge{x \in \RR}{x \geq 0}$, and the set of the positive real numbers by $\RPP := \menge{x \in \RR}{x >0}$.
The notation $A\colon X\rightrightarrows X$ indicates that $A$ is a set-valued operator on $X$. 

Given an operator $A$ on $X$, its \emph{domain} is denoted by $\dom A :=\menge{x\in X}{Ax\neq \varnothing}$, its \emph{range} by $\ran A :=A(X)$, its \emph{graph} by $\gra A :=\menge{(x,u)\in X\times X}{u\in Ax}$, its set of \emph{zeros} by $\zer A :=\menge{x\in X}{0\in Ax}$, and its \emph{fixed point} set by $\Fix A :=\menge{x\in X}{x\in Ax}$. The \emph{inverse} of $A$, denoted by $A^{-1}$, is the operator with graph $\gra A^{-1} :=\menge{(u,x)\in X\times X}{u\in Ax}$.
Recall from \cite[Definition~3.1]{DP18} that an operator $A\colon X\rightrightarrows X$ is said to be \emph{$\alpha$-monotone} if $\alpha\in \RR$ and
\begin{equation}
\forall x,y\in\dom A,\quad
\scal{x-y}{Ax-Ay}\geq\alpha\|x-y\|^2.
\end{equation}
In this case, we say that $A$ is \emph{monotone} if $\alpha =0$, \emph{strongly monotone} if $\alpha >0$, and \emph{weakly monotone} if $\alpha <0$. The operator $A$ is said to be \emph{maximally $\alpha$-monotone} if it is $\alpha$-monotone and there is no $\alpha$-monotone operator $B\colon X\rightrightarrows X$ such that $\gra B$ properly contains $\gra A$.   

The \emph{resolvent} and the \emph{reflected resolvent} of $A\colon X\rightrightarrows X$ are respectively defined by
\begin{equation}
J_A:= (\Id+ A)^{-1} \quad\text{and}\quad R_A :=2J_A -\Id,
\end{equation}
where $\Id$ is the identity operator. 
We conclude this section by an elementary formula for computing the resolvent of special composition via resolvents of its components.

\begin{proposition}[resolvent of composition]
\label{p:transform}
Let $A\colon X\rightrightarrows X$, $q, r\in X$, $\theta\in \RPP$, and $\sigma\in \RR$. Define $\bar{A} :=A\circ(\theta\Id-q)+\sigma\Id-r$ and let $\gamma\in \RPP$. Then the following hold:
\begin{enumerate}
\item\label{p:transform_mono}
$A$ is (resp. maximally) $\alpha$-monotone if and only if $\bar{A}$ is (resp. maximally) $(\theta\alpha+\sigma)$-monotone.
\item\label{p:transform_resol} 
If $1+\gamma\sigma\neq 0$, then
\begin{equation}
J_{\gamma\bar{A}} =\frac{1}{\theta}\left(J_{\frac{\gamma\theta}{1+\gamma\sigma}A}\circ \left(\frac{\theta}{1+\gamma\sigma}\Id+\frac{\gamma\theta}{1+\gamma\sigma}r-q\right) +q\right); 
\end{equation}
and if, in addition, $A$ is maximally $\alpha$-monotone and $1+\gamma(\theta\alpha+\sigma) >0$, then $J_{\gamma\bar{A}}$ and $J_{\frac{\gamma\theta}{1+\gamma\sigma}A}$ are single-valued and have full domain.
\end{enumerate}
\end{proposition}
\begin{proof} 
\ref{p:transform_mono}: This is straightforward from the definition.

\ref{p:transform_resol}: We note that $(\theta\Id-q)^{-1} =\frac{1}{\theta}(\Id+q)$, that $(T-z)^{-1} =T^{-1}\circ(\Id+z)$, and that $(\alpha T)^{-1} =T^{-1}\circ(\frac{1}{\alpha}\Id)$ for any operator $T$, any $z\in X$, and any $\alpha\in \RR\smallsetminus\{0\}$. Using these facts yields
\begin{subequations}
\begin{align}
J_{\gamma\bar{A}} &=\Big((1+\gamma\sigma)\Id +\gamma A\circ(\theta\Id-q) -\gamma r\Big)^{-1}\\
&=\left(\Big(\frac{1+\gamma\sigma}{\theta}(\Id+q) +\gamma A\Big)\circ (\theta\Id-q)\right)^{-1}\circ (\Id+\gamma r)\\
&=(\theta\Id-q)^{-1}\circ \left(\frac{1+\gamma\sigma}{\theta}\Id+\gamma A +\frac{1+\gamma\sigma}{\theta}q\right)^{-1}\circ (\Id+\gamma r)\\
&=(\theta\Id-q)^{-1}\circ \left(\frac{1+\gamma\sigma}{\theta}\Big(\Id+\frac{\gamma\theta}{1+\gamma\sigma}A\Big)\right)^{-1}\circ \left(\Id-\frac{1+\gamma\sigma}{\theta}q\right)\circ (\Id+\gamma r)\\
&=\frac{1}{\theta}(\Id+q)\circ \left(\Id+\frac{\gamma\theta}{1+\gamma\sigma} A\right)^{-1}\circ\left(\frac{\theta}{1+\gamma\sigma}\Id\right)\circ \left(\Id+\gamma r-\frac{1+\gamma\sigma}{\theta}q\right)\\
&=\frac{1}{\theta}(\Id+q)\circ J_{\frac{\gamma\theta}{1+\gamma\sigma}A}\circ \left(\frac{\theta}{1+\gamma\sigma}\Id+\frac{\gamma\theta}{1+\gamma\sigma}r-q\right)\\
&=\frac{1}{\theta}\left(J_{\frac{\gamma\theta}{1+\gamma\sigma}A}\circ \left(\frac{\theta}{1+\gamma\sigma}\Id+\frac{\gamma\theta}{1+\gamma\sigma}r-q\right) +q\right).
\end{align}
\end{subequations}
Since $A$ is maximally $\alpha$-monotone, $\bar{A}$ is maximally $(\theta\alpha+\sigma)$-monotone. Now, since $1+\gamma(\theta\alpha+\sigma) >0$, \cite[Proposition~3.4]{DP18} implies the conclusion.
\end{proof}

\section{Main results}
\label{s:main}

In this section, let $A, B\colon X\rightrightarrows X$, $\omega\in \RPP$, and $r\in X$.
We present a flexible approach to the computation of the resolvent at $r$ of the scaled sum $\omega(A+B)$, that is to
\begin{equation}
\text{compute $J_{\omega(A+B)}(r)$}. 
\end{equation}
Our analysis relies on the observation that this problem can be reformulated into the problem of finding a zero of the sum of two suitable operators. Indeed, when $r\in \dom J_{\omega(A+B)} =\ran\left(\Id+\omega(A+B)\right)$, we have by definition that
\begin{equation}
x\in J_{\omega(A+B)}(r) 
\iff r\in x +\omega(A+B)x
\iff 0\in (A+B)x +\frac{1}{\omega}x -\frac{1}{\omega}r.
\end{equation}
By writing $\frac{1}{\omega} =\sigma+\tau$ and $\frac{1}{\omega}r =r_A+r_B$, the last inclusion is equivalent to 
\begin{equation}
0\in (A+\sigma\Id-r_A)x +(B+\tau\Id-r_B)x, 
\end{equation}
which leads to finding a zero of the sum of two new operators $A+\sigma\Id-r_A$ and $B+\tau\Id-r_B$.

Based on this observation, we proceed with a more general formulation. Given $\theta\in\RPP$ and $q\in X$, we take $(\sigma,\tau)\in\RR^2$ and $(r_A,r_B)\in X^2$ satisfying 
\begin{equation}
\label{e:parameters}
\sigma+\tau =\frac{\theta}{\omega} \quad\text{and}\quad r_A+r_B =\frac{1}{\omega}(q+r),
\end{equation} 
and define 
\begin{equation}
\label{e:A'B'}
A_\sigma :=A\circ(\theta\Id-q)+\sigma\Id-r_A \quad\text{and}\quad B_\tau :=B\circ(\theta\Id-q)+\tau\Id-r_B.
\end{equation}
Now, we will derive the formula for the resolvent of the scaled sum via zeros of the sum of these newly defined operators.
\begin{proposition}[resolvent via zeros of sum of operators]
\label{p:resol=zer}
Suppose $r\in \ran\left(\Id+\omega(A+B)\right)$. Then
\begin{equation}
J_{\omega(A+B)}(r) =\theta\zer(A_\sigma+B_\tau)-q.
\end{equation}
Consequently, if $A_\sigma+B_\tau$ is strongly monotone, then $J_{\omega(A+B)}(r)$ and $\zer(A_\sigma+B_\tau)$ are singletons.
\end{proposition}
\begin{proof}
For every $z\in X$, we derive from \eqref{e:parameters} and \eqref{e:A'B'} that
\begin{subequations}
\begin{align} 
\theta z-q\in J_{\omega(A+B)}(r) 
&\iff r\in (\theta z-q)+\omega(A+B)(\theta z-q)\\
&\iff 0\in (A+B)(\theta z-q) +\frac{\theta}{\omega}z -\frac{1}{\omega}(q+r)\\
&\iff 0\in (A+B)(\theta z-q) +(\sigma+\tau)z -(r_A+r_B)\\
&\iff 0\in \big(A(\theta z-q)+\sigma z-r_A\big)+ \big(B(\theta z-q)+\tau z-r_B\big)\\
&\iff z\in \zer(A_\sigma+B_\tau).
\end{align}
\end{subequations}
The remaining conclusion follows from \cite[Proposition~23.35]{BC17}.
\end{proof}

The new operators $A_\sigma$ and $B_\tau$ along with Proposition~\ref{p:resol=zer} allow for the flexibility in chosing $(\sigma,\tau)$ and $(r_A,r_B)$ as one can decide the values of these parameters as long as \eqref{e:parameters} is satisfied. We are now ready for our main result.

\begin{theorem}[resolvent of sum of $\alpha$- and $\beta$-monotone operators]
\label{t:resol}
Suppose that $A$ and $B$ are respectively maximally $\alpha$- and $\beta$-monotone with $\alpha+\beta >-1/\omega$, that $r\in \ran\left(\Id+\omega(A+B)\right)$, and that $(\sigma,\tau)$ satisfies
\begin{equation}
\label{e:sigma-tau}
\theta\alpha+\sigma>0 \quad\text{and}\quad \theta\beta+\tau\geq 0.
\end{equation}
Let $\gamma\in \RPP$ be such that $1+\gamma\sigma\neq 0$ and $1+\gamma\tau\neq 0$. Given any $\kappa\in \left]0,1\right]$ and $x_0\in X$, define the sequence $(x_n)_\nnn$ by  
\begin{equation}
\label{e:x+}
\forall\nnn, \quad x_{n+1} :=(1-\kappa)x_n +\kappa(2J_{\gamma B_\tau}-\Id)(2J_{\gamma A_\sigma}-\Id)x_n
\end{equation} 
with explicit formulas
\begin{subequations}
\label{e:explicit}
\begin{align}
J_{\gamma A_\sigma} &=\frac{1}{\theta}\left(J_{\frac{\gamma\theta}{1+\gamma\sigma}A}\circ \left(\frac{\theta}{1+\gamma\sigma}\Id+\frac{\gamma\theta}{1+\gamma\sigma}r_A-q\right) +q\right)\\
\text{and~} J_{\gamma B_\tau} &=\frac{1}{\theta}\left(J_{\frac{\gamma\theta}{1+\gamma\tau}B}\circ \left(\frac{\theta}{1+\gamma\tau}\Id+\frac{\gamma\theta}{1+\gamma\tau}r_B-q\right) +q\right).
\end{align}
\end{subequations}
Then $J_{\omega(A+B)}(r) =J_{\frac{\gamma\theta}{1+\gamma\sigma}A}\left(\frac{\theta}{1+\gamma\sigma}\overline{x}+\frac{\gamma\theta}{1+\gamma\sigma}r_A-q\right)$ with $\overline{x}\in \Fix(2J_{\gamma B_\tau}-\Id)(2J_{\gamma A_\sigma}-\Id)$ and the following hold:
\begin{enumerate}
\item\label{t:resol_strong} 
$\left(J_{\frac{\gamma\theta}{1+\gamma\sigma}A}\left(\frac{\theta}{1+\gamma\sigma}x_n+\frac{\gamma\theta}{1+\gamma\sigma}r_A-q\right)\right)_\nnn$ converges strongly to $J_{\omega(A+B)}(r)$.
\item\label{t:resol_weak} 
If $\kappa <1$, then $(x_n)_\nnn$ converges weakly to $\overline{x}$.
\item\label{t:resol_linear} 
If $A$ is Lipschitz continuous, then the convergences in \ref{t:resol_strong} and \ref{t:resol_weak} are linear.
\end{enumerate}  
\end{theorem}
\begin{proof}
We first note that the existence of $(\sigma,\tau)\in \RR^2$ satisfying \eqref{e:parameters} and \eqref{e:sigma-tau} is ensured since $\alpha+\beta >-1/\omega$. Next, \eqref{e:sigma-tau} implies that $1+\gamma(\theta\alpha+\sigma) >0$ and $1+\gamma(\theta\beta+\tau) >0$. Using Proposition~\ref{p:transform}, we derive that $A_\sigma$ and $B_\tau$ are respectively maximally $(\theta\alpha+\sigma)$- and $(\theta\beta+\tau)$-monotone and that $J_{\gamma A_\sigma}$ and $J_{\gamma B_\tau}$ are single-valued and have full domain. It then follows from \eqref{e:sigma-tau} that $A_\sigma$ and also $A_\sigma+B_\tau$ are strongly monotone. By Proposition~\ref{p:resol=zer} and \cite[Proposition~26.1(iii)(b)]{BC17}, $\zer(A_\sigma+B_\tau) =\{J_{\gamma A_\sigma}(\overline{x})\}$ with $\overline{x}\in \Fix(2J_{\gamma B_\tau}-\Id)(2J_{\gamma A_\sigma}-\Id)$ and $J_{\omega(A+B)}(r) =\theta J_{\gamma A_\sigma}(\overline{x}) -q$. 

Now, Proposition~\ref{p:transform}\ref{p:transform_resol} implies \eqref{e:explicit}, which yields
\begin{equation}
\label{e:Asigma-A}
\theta J_{\gamma A_\sigma} -q =J_{\frac{\gamma\theta}{1+\gamma\sigma}A}\circ \left(\frac{\theta}{1+\gamma\sigma}\Id+\frac{\gamma\theta}{1+\gamma\sigma}r_A-q\right).
\end{equation}
Therefore,
\begin{equation}
\label{e:resol=shadow}
J_{\omega(A+B)}(r) =\theta J_{\gamma A_\sigma}(\overline{x}) -q =J_{\frac{\gamma\theta}{1+\gamma\sigma}A}\left(\frac{\theta}{1+\gamma\sigma}\overline{x}+\frac{\gamma\theta}{1+\gamma\sigma}r_A-q\right).
\end{equation}    

\ref{t:resol_strong}: By apply \cite[Theorem~26.11(iv)(b)]{BC17} with all $\lambda_n =\kappa$ if $\kappa <1$ and apply \cite[Proposition~26.13]{BC17} if $\kappa =1$, we obtain that $J_{\gamma A_\sigma}(x_n)\to J_{\gamma A_\sigma}(\overline{x})$. Now combine with \eqref{e:Asigma-A} and \eqref{e:resol=shadow}.

\ref{t:resol_weak}: Again apply \cite[Theorem~26.11]{BC17} with all $\lambda_n =\kappa$.

\ref{t:resol_linear}: Assume that $A$ is Lipschitz continuous with constant $\ell$. It is straightforward to see that $A_\sigma$ is Lipschitz continuous with constant $(\theta\ell+|\sigma|)$. The conclusion follows from \cite[Theorem~4.8]{DP18} with $\lambda =\mu =2$ and $\delta =\gamma$.  
\end{proof}

\begin{remark}
\label{r:resol}
Some remarks regarding Theorem~\ref{t:resol} are in order.
\begin{enumerate}
\item 
If $A$ is $\alpha$-monotone with $\alpha\in \RP$ and maximally monotone, then it is maximally $\alpha$-monotone (see \cite[Lemma~3.2]{DP18}). In \cite[Proposition~26.1(iii), Theorem~26.11, and Proposition~26.13]{BC17}, the maximal monotonicity assumption is needed only to ensure the full domain of resolvents, which already holds due to our assumptions on the choice of $\gamma$ and that $A$ and $B$ are respectively maximally $\alpha$- and $\beta$-monotone.
\item\label{r:resol_max} 
Under the assumptions made, $A+B$ is $(\alpha+\beta)$-monotone but not necessarily maximal. If in addition $A+B$ is indeed maximally $(\alpha+\beta)$-monotone, then $J_{\omega(A+B)}$ has full domain
by \cite[Proposition~3.4(ii)]{DP18}; thus, the condition on $r$ can be removed.
\item 
The iterative scheme \eqref{e:x+} is the Douglas--Rachford algorithm if $\kappa =1/2$ and the Peaceman--Rachford algorithm if $\kappa =1$. For a more general version of \eqref{e:x+}, we refer the readers to \cite{DP18}; see also \cite{DP18mor,DP18jogo}.
\item 
If the condition \eqref{e:sigma-tau} is replaced by 
\begin{equation}
\alpha+\sigma\geq 0 \quad\text{and}\quad \beta+\tau >0,
\end{equation}
then the conclusions of Theorem~\ref{t:resol}\ref{t:resol_weak}--\ref{t:resol_linear} still hold true, while Theorem~\ref{t:resol}\ref{t:resol_strong} only holds for $\kappa <1$; see also \cite[Theorem~2.1(ii) and Remark~2.2(iv)]{Com09}.
\item One can simply choose $\theta =1$ and $q =0$, in which case, \eqref{e:explicit} reduces to
\begin{equation}
J_{\gamma A_\sigma} =J_{\frac{\gamma}{1+\gamma\sigma}A}\circ\frac{1}{1+\gamma\sigma}(\Id+\gamma r_A)
\quad\text{and}\quad
J_{\gamma B_\tau} =J_{\frac{\gamma}{1+\gamma\tau}B}\circ\frac{1}{1+\gamma\tau}(\Id+\gamma r_B).
\end{equation} 
\item 
When $A$ and $B$ are maximally monotone, i.e., $\alpha =\beta =0$, \eqref{e:sigma-tau} is satisfied whenever $\sigma >0$ and $\tau\geq 0$. One thus can choose for instance $\sigma =\tau =\frac{\theta}{2\omega}$.
\item 
It is always possible to find $\gamma\in\RPP$ satisfying even $1+\gamma\sigma >0$ and $1+\gamma\tau >0$. In fact, these inequalities are automatic regardless of $\gamma\in\RPP$ as long as $\sigma$ and $\tau$ are both nonnegative.
\end{enumerate}
\end{remark}

When $A$ and $B$ are maximally monotone, the following result gives an iterative method for computing the resolvent of $A+B$ in which each iteration \emph{only} requires the computations in terms of $J_A$ and $J_B$.  
\begin{proposition}[resolvent of sum of two maximally monotone operators]
\label{p:maxmono}
Suppose that $A$ and $B$ are maximally monotone, that $\omega\neq 1/2$, and that $r\in \ran\left(\Id+\omega(A+B)\right)$. Define 
\begin{subequations}
\begin{align}
\bar{A} &:=\frac{2\omega}{\theta(2\omega-1)}A\circ(\theta\Id-q) +\frac{1}{\theta(2\omega-1)}(\theta\Id-q-r)\\ 
\text{and~}
\bar{B} &:=\frac{2\omega}{\theta(2\omega-1)}B\circ(\theta\Id-q) +\frac{1}{\theta(2\omega-1)}(\theta\Id-q-r).
\end{align}
\end{subequations}
Let $\kappa\in \left]0,1\right]$, let $x_0\in X$, and define the sequence $(x_n)_\nnn$ by
\begin{equation}
\forall\nnn, \quad x_{n+1} :=(1-\kappa)x_n +\kappa(2J_{\bar{B}}-\Id)(2J_{\bar{A}}-\Id)x_n
\end{equation} 
with explicit formulas
\begin{subequations}
\label{e:explicit'}
\begin{align}
J_{\bar{A}} &=\frac{1}{\theta}\left(J_A\circ\left(\Big(1-\frac{1}{2\omega}\Big)(\theta\Id-q) +\frac{1}{2\omega}r\right) +q\right)\\
\text{and~} 
J_{\bar{B}} &=\frac{1}{\theta}\left(J_B\circ\left(\Big(1-\frac{1}{2\omega}\Big)(\theta\Id-q) +\frac{1}{2\omega}r\right) +q\right). 
\end{align}
\end{subequations}
Then $J_{\omega(A+B)}(r) =J_A\left((1-\frac{1}{2\omega})(\theta\overline{x}-q)+\frac{1}{2\omega}r\right)$ with $\overline{x}\in \Fix(2J_{\bar{B}}-\Id)(2J_{\bar{A}}-\Id)$ and the following hold:
\begin{enumerate}
\item
$\left(J_A\left((1-\frac{1}{2\omega})(\theta x_n-q)+\frac{1}{2\omega}r\right)\right)_\nnn$ converges strongly to $J_{\omega(A+B)}(r)$.
\item
If $\kappa <1$, then $(x_n)_\nnn$ converges weakly to $\overline{x}$.
\item 
If $A$ is Lipschitz continuous, then the above convergences are linear.
\end{enumerate} 
\end{proposition}
\begin{proof}
Choosing
\begin{equation}
\sigma =\tau =\frac{\theta}{2\omega} >0,\quad r_A =r_B =\frac{1}{2\omega}(q+r), \quad\text{and}\quad \gamma =\frac{2\omega}{\theta(2\omega-1)} >0, 
\end{equation}
we have that \eqref{e:parameters} is satisfied and that
\begin{equation}
A_\sigma =A\circ(\theta\Id-q) +\frac{1}{2\omega}(\theta\Id-q-r)
\text{~and~}
B_\tau =B\circ(\theta\Id-q) +\frac{1}{2\omega}(\theta\Id-q-r),
\end{equation}
which yields $\gamma A_\sigma =\bar{A}$ and $\gamma B_\tau =\bar{B}$. Since $1+\gamma\sigma =1+\gamma\theta/(2\omega) =2\omega/(2\omega-1) =\gamma\theta$, we get \eqref{e:explicit'} from \eqref{e:explicit}. Now apply Theorem~\ref{t:resol} with $\alpha =\beta =0$.
\end{proof}

Thanks to the flexibility of parameters, our results recapture the formulation and convergence analysis of recent methods. In particular, Corollaries~\ref{c:q=-r} and \ref{c:q=cr} are in the spirit of \cite[Theorem~3.1]{AC18} and \cite[Theorem~3.2]{ABMW18}, respectively.

\begin{corollary}
\label{c:q=-r}
Let $\eta\in \left]0,1\right[$ and $\gamma\in \RPP$. Suppose that $A$ and $B$ are maximally monotone and that $r\in \ran\left(\Id+\frac{\gamma}{2(1-\eta)}(A+B)\right)$. 
Let $\kappa\in \left]0,1\right]$, let $x_0\in X$, and define the sequence $(x_n)_\nnn$ by
\begin{equation}
\forall\nnn, \quad x_{n+1} :=(1-\kappa)x_n +\kappa\Big(2\eta J_{\gamma B}(\Id+r) -2\eta r -\Id\Big)\Big(2\eta J_{\gamma A}(\Id+r) -2\eta r -\Id\Big)x_n.
\end{equation} 
Then $J_{\frac{\gamma}{2(1-\eta)}(A+B)}(r) =J_{\gamma A}(\overline{x}+r)$ with $\overline{x}\in \Fix(2\eta J_{\gamma B}(\Id+r) -2\eta r -\Id)(2\eta J_{\gamma A}(\Id+r) -2\eta r -\Id)$ and the following hold:
\begin{enumerate}
\item
$\left(J_{\gamma A}(x_n+r)\right)_\nnn$ converges strongly to $J_{\frac{\gamma}{2(1-\eta)}(A+B)}(r)$.
\item
If $\kappa <1$, then $(x_n)_\nnn$ converges weakly to $\overline{x}$.
\item
If $A$ is Lipschitz continuous, then the above convergences are linear.
\end{enumerate} 
\end{corollary}
\begin{proof}
Let $\omega =\frac{\gamma}{2(1-\eta)}$, $\theta =\frac{1}{\eta}$, $\sigma =\tau =\frac{\theta}{2\omega} =\frac{1-\eta}{\gamma\eta}$, $q =-r$, and $r_A =r_B =0$. Then \eqref{e:parameters} is satisfied,  
\begin{equation}
A_\sigma =A\circ\left(\frac{1}{\eta}\Id +r\right) +\frac{1-\eta}{\gamma\eta}\Id \quad\text{and}\quad
B_\tau =B\circ\left(\frac{1}{\eta}\Id +r\right) +\frac{1-\eta}{\gamma\eta}\Id. 
\end{equation} 
Noting $1+\gamma\sigma =1+(1-\eta)/\eta =1/\eta =\theta$, we have from \eqref{e:explicit} that
\begin{equation}
J_{\gamma A_\sigma} =\eta\left(J_{\gamma A}(\Id +r) -r\right) \quad\text{and}\quad 
J_{\gamma B_\tau} =\eta\left(J_{\gamma B}(\Id +r) -r\right).
\end{equation}
Applying Theorem~\ref{t:resol} with $\alpha =\beta =0$ completes the proof.
\end{proof}

\begin{corollary}
\label{c:q=cr}
Suppose that $A$ and $B$ are maximally monotone and that $A+B$ is also maximally monotone. Let $\eta\in \left]0,1\right[$, $\kappa\in \left]0,1\right]$, $x_0\in X$, and define the sequence $(x_n)_\nnn$ by
\begin{equation}
\forall\nnn, \quad x_{n+1} :=(1-\kappa)x_n +\kappa\Big(2\eta J_B +2(1-\eta)r -\Id\Big)\Big(2\eta J_A +2(1-\eta)r -\Id\Big)x_n.
\end{equation} 
Then $J_{\frac{1}{2(1-\eta)}(A+B)}(r) =J_A(\overline{x})$ with $\overline{x}\in \Fix(2\eta J_B +2(1-\eta)r -\Id)(2\eta J_A +2(1-\eta)r -\Id)$ and the following hold:
\begin{enumerate}
\item
$(J_A(x_n))_\nnn$ converges strongly to $J_{\frac{1}{2(1-\eta)}(A+B)}(r)$.
\item
If $\kappa <1$, then $(x_n)_\nnn$ converges weakly to $\overline{x}$.
\item
If $A$ is Lipschitz continuous, then the above convergences are linear.
\end{enumerate} 
\end{corollary}
\begin{proof}
Apply Proposition~\ref{p:maxmono} with $\omega =\frac{1}{2(1-\eta)}$, $\theta =\frac{1}{\eta}$, and $q =\frac{1-\eta}{\eta}r =\frac{1}{2\omega-1}r$ and note that $J_{\frac{1}{2(1-\eta)}(A+B)}$ has full domain due to Remark~\ref{r:resol}\ref{r:resol_max}.
\end{proof}

\section{Applications}
\label{s:appl}

In this section, we provide transparent applications of our result to computing the proximity operator of the sum of two weakly convex functions and to finding the closest point in the intersection of closed convex sets. 

We recall that a function $f\colon X\to \left]-\infty,+\infty\right]$ is \emph{proper} if $\dom f :=\menge{x\in X}{f(x) <+\infty}\neq \varnothing$ and \emph{lower semicontinuous} if $\forall x\in \dom f$, $f(x)\leq \liminf_{z\to x} f(z)$.
The function $f$ is said to be \emph{$\alpha$-convex} for some $\alpha\in \RR$ if $\forall x,y\in\dom f$, $\forall\kappa\in \left]0,1\right[$,
\begin{equation}
f((1-\kappa) x+\kappa y) +\frac{\alpha}{2}\kappa(1-\kappa)\|x-y\|^2\leq (1-\kappa)f(x)+\kappa f(y).
\end{equation}
We say that $f$ is \emph{convex} if $\alpha =0$, \emph{strongly convex} if $\alpha >0$, and \emph{weakly convex} if $\alpha <0$.

Let $f\colon X\to \left]-\infty,+\infty\right]$ be proper. The \emph{Fr\'echet subdifferential} of $f$ at $x$ is given by 
\begin{equation}
\widehat{\partial}f(x) :=\Menge{u\in X}{\liminf_{z\to x}\frac{f(z)-f(x)-\scal{u}{z-x}}{\|z-x\|}\geq 0}.
\end{equation}
It is known that if $f$ is differentiable at $x$, then $\widehat{\partial}f(x) =\{\nabla f(x)\}$, and that if $f$ is convex, then the Fr\'echet subdifferential coincides with the \emph{convex subdifferential}, i.e.,
\begin{equation}
\widehat{\partial}f(x) =\partial f(x) :=\menge{u\in X}{\forall z\in X,\ f(z)-f(x)\geq \scal{u}{z-x}}.
\end{equation}
The \emph{proximity operator} of $f$ with parameter $\gamma\in \RPP$ is the mapping $\prox_{\gamma f}\colon X\rightrightarrows X$ defined by
\begin{equation}
\forall x\in X, \quad \prox_{\gamma f}(x) :=\argmin_{z\in X}\left(f(z)+\frac{1}{2\gamma}\|z-x\|^2\right).
\end{equation}
Given a nonempty closed subset $C$ of $X$, the \emph{indicator function} $\iota_C$ of $C$ is defined by $\iota_C(x) =0$ if $x\in C$ and $\iota_C(x) =+\infty$ if $x\notin C$.
It is clear that $\prox_{\gamma \iota_C} =P_C$, where $P_C\colon X\rightrightarrows C$ is the \emph{projector} onto $C$ given by 
\begin{equation}
\forall x\in X,\quad P_Cx :=\argmin_{c\in C} \|x-c\|.
\end{equation}
If $C$ is convex, then the \emph{normal cone} to $C$ is the operator $N_C\colon X\rightrightarrows X$ defined by 
\begin{equation}
\forall x\in X, \quad N_C(x) :=\begin{cases}
\menge{u\in X}{\forall c\in C,\ \scal{u}{c-x}\leq 0} &\text{if~} x\in C,\\
\varnothing &\text{otherwise}.
\end{cases}
\end{equation}

\begin{lemma}
\label{l:resol=prox}
Let $f\colon X\to \left]-\infty,+\infty\right]$ and $g\colon X\to \left]-\infty,+\infty\right]$ be proper, lower semicontinuous, and respectively $\alpha$- and $\beta$-convex, let $\omega\in \RPP$, and let $r\in \ran(\Id+\omega(\widehat{\partial}f+\widehat{\partial}g))$. Suppose that $\alpha+\beta >-1/\omega$. Then 
\begin{equation}
\label{e:resol=prox}
J_{\omega(\widehat{\partial}f+\widehat{\partial}g)}(r)= J_{\omega\widehat{\partial}(f+g)}(r) =\prox_{\omega(f+g)}(r).
\end{equation}
Consequently, if $C$ and $D$ are closed convex subsets of $X$ with $C\cap D\neq \varnothing$ and $r\in \ran(\Id+N_C+N_D)$, then 
\begin{equation}
J_{N_C+N_D}(r) =P_{C\cap D}(r).
\end{equation}
\end{lemma}
\begin{proof}
On the one hand, noting that $\ran(\Id+\omega(\widehat{\partial}f+\widehat{\partial}g)) =\dom J_{\gamma(\widehat{\partial}f+\widehat{\partial}g)}$ and that $\widehat{\partial}f+\widehat{\partial}g\subseteq \widehat{\partial}(f+g)$, we have for any $p\in X$ that
\begin{subequations}
\label{e:noCQ}
\begin{align}
p\in J_{\omega(\widehat{\partial}f+\widehat{\partial}g)}(r) 
&\iff r\in p+\omega(\widehat{\partial}f+\widehat{\partial}g)(p)\\
&\implies r\in p+\omega\widehat{\partial}(f+g)(p)\\
&\iff p\in J_{\widehat{\partial}(f+g)}(r). 
\end{align}
\end{subequations} 
On the other hand, it follows from, e.g., \cite[Lemma~5.3]{DP18} that $f+g$ is $(\alpha+\beta)$-convex. Since $1+\omega(\alpha+\beta) >0$, we learn from \cite[Lemma~5.2]{DP18} that $\prox_{\omega(f+g)} =J_{\omega\widehat{\partial}(f+g)}$ is single-valued and has full domain. Combining with \eqref{e:noCQ} implies \eqref{e:resol=prox}.

Now, since $C$ and $D$ are closed convex sets, $\iota_C$ and $\iota_D$ are convex functions, and therefore, $\widehat{\partial}\iota_C =\partial\iota_C =N_C$ and $\widehat{\partial}\iota_D =\partial\iota_D =N_D$. Applying \eqref{e:resol=prox} to $(f,g) =(\iota_C,\iota_D)$ and $\omega =1$ yields
\begin{equation}
J_{N_C+N_D}(r) =\prox_{\iota_C+\iota_D}(r) =\prox_{\iota_{C\cap D}}(r) =P_{C\cap D}(r),
\end{equation} 
which completes the proof. 
\end{proof}

We now derive some applications of Theorem~\ref{t:resol}. In what follows, $\theta\in \RPP$ and $q\in X$.
  
\begin{corollary}[proximity operator of sum of $\alpha$- and $\beta$-convex functions]
\label{c:prox}
Let $f\colon X\to \left]-\infty,+\infty\right]$ and $g\colon X\to \left]-\infty,+\infty\right]$ be proper, lower semicontinuous, and respectively $\alpha$- and $\beta$-convex, let $\omega\in \RPP$, and let $r\in \ran(\Id+\omega(\widehat{\partial}f+\widehat{\partial}g))$. Suppose that $\alpha+\beta >-1/\omega$ and let $(\sigma,\tau)\in \RR^2$ and $(r_f, r_g)\in X^2$ be such that $\sigma+\tau =\theta/\omega$, $r_f+r_g =(q+r)/\omega$,
\begin{equation}
\theta\sigma+\alpha>0 \quad\text{and}\quad \theta\beta+\tau\geq 0.
\end{equation}
Let $\gamma\in \RPP$ be such that $1+\gamma\sigma >0$ and $1+\gamma\tau >0$. Given any $\kappa\in \left]0,1\right]$ and $x_0\in X$, define the sequence $(x_n)_\nnn$ by 
\begin{equation}
\forall\nnn, \quad x_{n+1} :=(1-\kappa)x_n +\kappa R_gR_f x_n,
\end{equation} 
where
\begin{subequations}
\begin{align}
R_f &:=\frac{2}{\theta}\left(\prox_{\frac{\gamma\theta}{1+\gamma\sigma}f}\circ \left(\frac{\theta}{1+\gamma\sigma}\Id+\frac{\gamma\theta}{1+\gamma\sigma}r_f-q\right) +q\right) -\Id\\ 
\text{~and~}
R_g &:=\frac{2}{\theta}\left(\prox_{\frac{\gamma\theta}{1+\gamma\tau}g}\circ \left(\frac{\theta}{1+\gamma\tau}\Id+\frac{\gamma\theta}{1+\gamma\tau}r_g-q\right) +q\right) -\Id.
\end{align}
\end{subequations}
Then $\prox_{\omega(f+g)}(r) =\prox_{\frac{\gamma\theta}{1+\gamma\sigma}f}\left(\frac{\theta}{1+\gamma\sigma}\overline{x}+\frac{\gamma\theta}{1+\gamma\sigma}r_f-q\right)$ with $\overline{x}\in \Fix R_gR_f$ and the following hold:
\begin{enumerate}
\item
$\left(\prox_{\frac{\gamma\theta}{1+\gamma\sigma}f}\left(\frac{\theta}{1+\gamma\sigma}x_n+\frac{\gamma\theta}{1+\gamma\sigma}r_f-q\right)\right)_\nnn$ converges strongly to $\prox_{\omega(f+g)}(r)$.
\item 
If $\kappa <1$, then $(x_n)_\nnn$ converges weakly to $\overline{x}$.
\item
If $f$ is differentiable with Lipschitz continuous gradient, then the above convergences are linear.
\end{enumerate}  
\end{corollary}
\begin{proof}
By assumption, \cite[Lemma~5.2]{DP18} implies that $\widehat{\partial}f$ and $\widehat{\partial}g$ are respectively maximally $\alpha$- and $\beta$-monotone, and that
\begin{equation}
J_{\frac{\gamma}{1+\gamma\sigma}\widehat{\partial}f} =\prox_{\frac{\gamma}{1+\gamma\sigma}f} \quad\text{and}\quad J_{\frac{\gamma}{1+\gamma\tau}\widehat{\partial}g} =\prox_{\frac{\gamma}{1+\gamma\tau}g}.
\end{equation}  
Next, Lemma~\ref{l:resol=prox} implies that $J_{\omega(\widehat{\partial}f+\widehat{\partial}g)}(r) =\prox_{\omega(f+g)}(r)$. 
The conclusion then follows from Theorem~\ref{t:resol} applied to $(A,B) =(\widehat{\partial}f,\widehat{\partial}g)$. 
\end{proof}

\begin{corollary}[projection onto intersection of two closed convex sets]
\label{c:bestapprox}
Let $C$ and $D$ be closed convex subsets of $X$ with $C\cap D\neq \varnothing$ and let $r\in \ran(\Id+N_C+N_D)$. Let $\sigma\in \RPP$, $\tau\in \RP$, and $(r_C,r_D)\in X^2$ with $r_C+r_D=(\sigma+\tau)(q+r)/\theta$. Let also $\gamma\in \RPP$, $\kappa\in \left]0,1\right]$, $x_0\in X$, and define the sequence $(x_n)_\nnn$ by
\begin{equation}
\forall\nnn, \quad x_{n+1} :=(1-\kappa)x_n +\kappa\bar{R}_D\bar{R}_C x_n,
\end{equation} 
where
\begin{subequations}
\begin{align}
\bar{R}_C &:=\frac{2}{\theta}\left(P_C\circ\left(\frac{\theta}{1+\gamma\sigma}\Id+\frac{\gamma\theta}{1+\gamma\sigma}r_C-q\right) +q\right) -\Id\\ 
\text{and~}
\bar{R}_D &:=\frac{2}{\theta}\left(P_D\circ\left(\frac{\theta}{1+\gamma\tau}\Id+\frac{\gamma\theta}{1+\gamma\tau}r_D-q\right) +q\right) -\Id.
\end{align}
\end{subequations}
Then $\left(P_C\left(\frac{\theta}{1+\gamma\sigma}x_n+\frac{\gamma\theta}{1+\gamma\sigma}r_C-q\right)\right)_\nnn$ converges strongly to $P_{C\cap D}(r) =P_C\left(\frac{\theta}{1+\gamma\sigma}\overline{x}+\frac{\gamma\theta}{1+\gamma\sigma}r_C-q\right)$ with $\overline{x}\in \Fix \bar{R}_D\bar{R}_C$. Furthermore, if $\kappa <1$, then $(x_n)_\nnn$ converges weakly to $\overline{x}$.
\end{corollary}
\begin{proof}
We first derive from \cite[Example~20.26]{BC17} that $N_C$ and $N_D$ are maximally monotone and from \cite[Example~23.4]{BC17} that
\begin{equation}
J_{\frac{\gamma}{1+\gamma\sigma}N_C} =J_{N_C} =P_C \quad\text{and}\quad J_{\frac{\gamma}{1+\gamma\tau}N_D} =J_{N_D} =P_D.
\end{equation} 
Setting $\omega :=\theta/(\sigma+\tau) >0$, we note that $r\in \ran(\Id+N_C+N_D) =\ran\big(\Id+\omega(N_C+N_D)\big)$ and from Lemma~\ref{l:resol=prox} that $J_{\omega(N_C+N_D)}(r) =J_{N_C+N_D}(r) =P_{C\cap D}(r)$. Now apply Theorem~\ref{t:resol} to $(A,B) =(N_C,N_D)$.
\end{proof}

As in the proof of Corollary~\ref{c:q=-r}, if we choose $\theta =\frac{1}{\eta}$ (with $\eta\in \left]0,1\right[$), $\sigma =\tau =\frac{1-\eta}{\gamma\eta}$, $q =-r$, and $r_C =r_D =0$, then Corollary~\ref{c:bestapprox} reduces to \cite[Corollary~3.1]{AC18}.

\section*{Acknowledgement}
MND was partially supported by Discovery Project 160101537 from the Australian Research Council. HMP was partially supported by a Research Fund from Autodesk, Inc.

\end{document}